\documentclass[preprints,article,accept,oneauthor,pdftex]{Definitions/mdpi}
\pdfoutput=1 

\firstpage{1}
\pubvolume{xx}
\issuenum{x}
\articlenumber{x}
\pubyear{2022}
\copyrightyear{2022}
\history{}

\usepackage{amsfonts}
\usepackage{tikz-cd}
\usepackage{physics}

\theoremstyle{plain}
\newtheorem{proposition}[lemma]{Proposition}
\newtheorem{theorem}[lemma]{Theorem}
\newtheorem{corollary}[lemma]{Corollary}

\newcommand{\fg}{\mathfrak{g}}
\newcommand{\RR}{\mathbb{R}}
\newcommand{\hatG}{\widehat{G}}
\newcommand{\hatfg}{\widehat{\fg}}
\newcommand{\eL}{\mathcal{L}}
\newcommand{\bigO}{\mathcal{O}}
\newcommand{\Ad}{\operatorname{Ad}}

\newcommand{\Aff}{\operatorname{Aff}}

\newcommand{\pair}[2]{\left\langle #1,#2\right\rangle}

\setitemize{parsep=6pt,itemsep=0pt,leftmargin=*,labelsep=5.5mm}
\setenumerate{parsep=6pt,itemsep=0pt,leftmargin=*,labelsep=5.5mm} 
\setlist[description]{itemsep=0mm}   
\usepackage{environ}
\NewEnviron{myequation}{%
\begin{equation} 
\scalebox{0.95}{$\BODY$}
\end{equation}
}

\Title{Homogeneous Symplectic Spaces and Central Extensions$\dagger$}

\newcommand{\orcidauthorA}{0000-0002-7287-3156} 

\Author{Andrew Beckett $^1$ orcid number:\orcidauthorA} 

\AuthorNames{Andrew Beckett}

\address{%
$^1$ \quad Maxwell Institute and School of Mathematics, The University
  of Edinburgh, James Clerk Maxwell Building, Peter Guthrie Tait Road,
  Edinburgh EH9 3FD, Scotland, United Kingdom; abeckett@ed.ac.uk}
\corres{Andrew Beckett}

\firstnote{Submitted to International Workshop on Bayesian Inference and Maximum Entropy Methods in Science and Engineering, IHP, Paris, July 18-22, 2022.} 

\abstract{We summarise recent work \cite{Beckett:2022wvo} on the classical result of Kirillov \cite{MR0412321,MR2069175} that any simply-connected homogeneous symplectic space of a connected group $G$ is a hamiltonian $\widehat{G}$-space for a one-dimensional central extension $\widehat{G}$ of $G$, and is thus (by a result of Kostant \cite{MR0294568}) a cover of a coadjoint orbit of $\widehat{G}$. We emphasise that existing proofs in the literature assume that $G$ is simply-connected and that this assumption can be removed by application of a theorem of Neeb \cite{MR1424633}. We also interpret Neeb's theorem as relating the integrability of one-dimensional central extensions of Lie algebras to the integrability of an associated Chevalley--Eilenberg 2-cocycle.}


\begin{document}

\section{Introduction}
\label{sec:Introduction}

In his influential book \emph{Structure of Dynamical Systems} \cite{MR0260238,MR1461545}, Souriau introduced the notion of an "elementary system" as the classical analogue of an elementary particle. An elementary system is defined by a phase space which has the structure of a homogeneous symplectic $G$-space; that is, a symplectic manifold equipped with a transitive action by a Lie group $G$, which we assume to be connected. Important examples of such spaces are the coadjoint orbits of $G$ itself, which have a natural $G$-invariant symplectic structure known as the Kirillov--Kostant--Souriau (KKS) structure.  

A related concept also due to Souriau is the moment map. A comoment map for a symplectic manifold $(M,\omega)$ under the action of a group $G$ is an assignment $\varphi:\fg\to C^\infty(M)$, $X\mapsto \varphi_X$  satisfying the equation $\imath_{\xi_X}\omega=d\varphi_X$, where $\imath$ is insertion into the first argument and $\xi_X$ is the fundamental vector field associated to $X$, ie. the velocity of the flow on $M$ generated by $X$. The moment map is defined as the dual to the comoment map, $\pair{\mu(p)}{X\rangle=\varphi_X(p)}$, where $p\in M$ and the angle brackets denote the dual pairing between $\fg$ and $\fg^*$. A (co)moment map always exists in particular if $M$ is simply-connected.

In the case of a homogeneous symplectic $G$-space, a moment map is in particular a symplectomorphism onto a coadjoint orbit $\mathcal{O}$. Since both the source and the target are also $G$-spaces, it is a natural question whether $\mu$ is equivariant with respect to the group action. If it is equivariant, $(M,\omega,\mu)$ is called a \emph{hamiltonian $G$-space} and a theorem of Kostant \cite{MR0294568} says that $\mu:M\to \mathcal{O}$ is an equivariant covering map. If $M$ is simply-connected, it is thus the universal cover of a coadjoint orbit.

More generally, we can define a map $\vartheta:M\times G\to\fg^*$ by $\textstyle \vartheta(p,g) = \Ad^*_g\mu(p) - \mu(g\cdot p)$ for $p\in M$ and $g\in G$, where $\Ad^*$ is the coadjoint representation. One can show that $\vartheta$ is in fact independent of $p$ and thus defines a map $\theta:G\to\fg^*$ by $\theta(g)=\vartheta(p,g)$ for any $p\in M$. This map satisfies the group cocycle condition $\textstyle \theta(g_1,g_2) = \Ad^*_{g_1}\theta(g_2) + \theta(g_1) $ and is thus called a \emph{symplectic cocycle}. The moment map can be chosen to be equivariant if and only if the cohomology class of $\theta$ is trivial.

Defining $c(X,Y)=\pair{d_e\theta(X)}{Y}$ for $X,Y\in\fg$ (where $e$ is the identity element of $G$), one finds that $c(X,Y)=\{\varphi_X,\varphi_Y\}-\varphi_{[X,Y]}$ where $\{\}$ is the Poisson bracket. In particular, $c$ is skew-symmetric and is in fact a Chevaley-Eilenberg 2-cocycle $c$ for $\fg$. Its cohomology class is trivial if and only if the cohomology class of $\theta$ is trivial. This cocycle $c$ defines a central extension of $\fg$ with base space $\hatfg=\fg\oplus\mathbb{R}$ and bracket $[(X,u),(Y,v)] = ([X,Y],c(X,Y))$. If $[c]=0$, this is isomorphic to the trivial central extension.

In the nontrivial case, it has long been believed that one can "integrate" this central extension of Lie algebras to a one-dimensional central extension $\hatG$ of the group $G$ and find that $M$ is a hamiltonian $\hatG$-space and thus a cover of a coadjoint orbit. Indeed, a commonly-cited result of Kirillov \cite{MR0412321,MR2069175} essentially states this, but Kirillov assumes earlier in both works that the group $G$ is simply-connected. The same assumption is made by Kostant \cite{MR0294568}, Chu \cite{MR342642} and Sternberg \cite{MR379759} when proving versions of this result. If one does not make this assumption, the integration of the central extension is not trivial. To our knowledge, there is no proof without this assumption in the literature, although there is some related work; de Saxcé and Vallée outline a procedure to produce a group extension using the symplectic cocycle $\theta$ and provide some examples but do not prove the result in generality \cite{MR2587386}, while Donato and Iglesias-Zemmour refer to the result and prove a diffeological version of it \cite{donato_iglesias-zemmour_2021}.

The question of necessary and sufficient conditions for the integrability of a one-dimensional Lie algebra central extension was first tackled by Tuynman and Wiegerinck \cite{MR948561} and then by Neeb \cite{MR1424633}, who improved on the earlier work by removing one of the two conditions provided by Tuynman and Wiegerinck. The remaining condition is interesting from the point of view of symplectic geometry because it essentially demands the existence of a moment map for the left-action of the group $G$ on itself with respect to the the unique left-invariant presymplectic form with value $c$ at the identity.

This conference paper is based on work with José Figueroa-O'Farrill \cite{Beckett:2022wvo} in which we show that the symplectic cocycle $\theta$ provides the moment map for the action of $G$ on itself required by Neeb, so that the required central extension $\hatG$ of $G$ does indeed exist and $M$ really is the cover of a coadjoint orbit of $\hatG$. In the process of removing the assumption that $G$ is simply-connected from Kirillov's theorem, we find that Neeb's integrability condition for one-dimensional central extensions can be rephrased in terms of the existence of a symplectic cocycle $\theta$ integrating the Chevally-Eilenberg cocycle $c$ whose cohomology class is the one corresponding to the algebra extension.

\section{Hamiltonian $G$-spaces are covers of coadjoint orbits}

We begin by expanding upon the claim in that hamiltonian $G$-spaces are covers of coadjoint orbits. As discussed in the introduction, a hamiltonian $G$-space is a triple $(M,\omega,\mu)$ where $(M,\omega)$ is a symplectic manifold with a homogeneous, symplectic action of $G$ and $\mu:M\to\fg^*$ is a moment map for the action of $G$ which is equivariant with respect to the coadjoint action on $\fg^*$; that is, for all $g\in G$, $\Ad^*_g\mu(p)=\mu(g\cdot p)$. If $G$ is connected (and hence $M$ is also connected), equivariance of the moment map is equivalent to the dual comoment map $\varphi:\fg\to C^\infty(M)$ being a Lie algebra homomorphism. Below we record the aforementioned theorem due to Kostant, which concerns morphisms of hamilonian spaces $(M,\omega,\mu) \to (M',\omega',\mu')$, ie. smooth maps $\phi: M \to M'$ such that $\phi^*\omega'= \omega$ and $\phi^* \mu' = \mu$.

\begin{theorem}[Kostant \cite{MR0294568}]\label{thm:kostant-covering}
  Let $\phi: M \to M'$ be a morphism of hamiltonian $G$-spaces
  $(M,\omega,\mu)$ and $(M',\omega',\mu')$, with $G$ a connected Lie
  group.  Then $\phi$ is $G$-equivariant and a covering map.
\end{theorem}

Recall that any coadjoint orbit $\bigO\subseteq\fg^*$ of $G$ has a natural $\Ad^*$-invariant symplectic structure $\omega_{\mathrm{KKS}}$ which is uniquely defined by the fact that at a point $\alpha\in\bigO$,

\begin{equation}
	(\omega_{\mathrm{KKS}})(\zeta_X,\zeta_Y)(\alpha) = \alpha(\comm{X}{Y})
\end{equation}
for all $X,Y\in\fg$, where $\zeta_X$ is the fundamental vector field for the coadjoint action on $\bigO$ associated to $X$. Then $(\bigO,\omega_{\mathrm{KKS}},i)$ is a hamiltonian $G$-space where $i:\bigO\to\fg^*$ is the inclusion map. If $(M,\omega,\mu)$ is a hamiltonian $G$-space with $G$ connected, the image of the moment map is a coadjoint orbit $\bigO \subset \fg^*$ and we may restrict the codomain so that $\mu : M \to \bigO$. One can then show the following.

\begin{proposition}\label{prop:moment-map-is-morphism}
  The moment map $\mu : M \to \bigO$ defines a morphism of hamiltonian
  $G$-spaces $(M,\omega,\mu)$ and $(\bigO,\omega_{\mathrm{KKS}},i)$.
\end{proposition}

There is an immediate corollary of the two results above.

\begin{corollary}\label{cor:covering}
  Let $G$ be a connected Lie group and $(M,\omega,\mu)$ a simply-connected hamiltonian $G$-space. Then $M$ is the universal cover of a coadjoint orbit $\bigO\in\fg^*$ with $G$-equivariant symplectic covering map $\mu:M\to \bigO$.
\end{corollary}

Our main result, Theorem~\ref{thm:homo-sym-ext}, can be viewed as a generalisation of this corollary to simply-connected homogeneous symplectic $G$-spaces $(M,\omega)$, for which moment maps always exist but cannot generally be chosen to be equivariant, the obstruction being the cohomology class of the symplectic cocycle $\theta$. The strategy in the proof of that theorem is to show that there exists a one-dimensional central extension $\hatG$ of $G$ and that $M$ has the structure of a hamiltonian $\hatG$-space, and then apply the corollary directly.

\section{Symplectic cocycles and one-dimensional central extensions of Lie algebras}

Throughout, let $G$ be a connected Lie group. A 1-cocycle of $G$ with values in the coadjoint representation is a smooth map $\theta:G\to \fg^*$ satisfying the cocycle condition

\begin{equation}\label{eq:theta-cocycle}
	\theta(g_1,g_2) = \Ad^*_{g_1}\theta(g_2) + \theta(g_1).
\end{equation}
The derivative of $\theta$ at the identity is a linear map $d_e\theta:\fg\to\fg^*$, and we say that $\theta$ is a \emph{symplectic (group) cocycle} if $\pair{d_e\theta(X)}{Y}=-\pair{d_e(Y)}{X}$ for all $X,Y\in\fg$. The symplectic cocycle of Souriau described in the introduction is just such an object. To a symplectic cocycle $\theta$, we associate a Chevalley--Eilenberg 2-cocycle $c$ of $\fg$ defined by $c(X,Y)=\pair{d_e\theta(X)}{Y}$ (the cocycle condition for $c$ follows from that for $\theta$). We say that $\theta$ \emph{integrates} $c$, and one can show that, since $G$ is connected, if there exists a symplectic cocycle $\theta$ of $G$ integrating a given 2-cocycle $c$ of $\fg$, it is unique, and that $[c]=0$ in Lie algebra cohomology if and only if $[\theta]=0$ in Lie group cohomology.

There is a one-to-one correspondence between $H^2(\fg)$, the second cohomology of $\fg$, and (equivalence classes of) one-dimensional central extensions of $\fg$; that is, short exact sequences of Lie algebras

\begin{equation}\label{eq:alg-ce}
\begin{tikzcd}
	0 \arrow[r] & \RR \arrow[r] & \hatfg \arrow[r] &
	\fg \arrow[r] & 0
\end{tikzcd}
\end{equation}
where the image of the map $\RR$ in $\hatfg$ is central. Concretely, to a 2-cocycle $c$ of $\fg$, we associate a Lie algebra with base space $\hatfg=\fg\oplus\mathbb{R}$ and bracket $[(X,u),(Y,v)] = ([X,Y],c(X,Y))$; the maps in the short exact sequence are the obvious inclusion and projection, and the central extension associated to some other cocycle $c'$ is equivalent if and only if $[c']=[c]$.

\section{Integrability of one-dimensional central extensions}

A one-dimensional central extension of a connected Lie group $G$ is a short exact sequence of Lie groups

\begin{equation}\label{eq:grp-ce}
\begin{tikzcd}
	1 \arrow[r] & K \arrow[r] & \hatG \arrow[r] &
	G \arrow[r] & 1
\end{tikzcd}
\end{equation}
where $K$ is a connected one-dimensional Lie group whose image in $\hatG$ is central. We will occasionally abuse language and simply say that $\hatG$  is a central extension of $G$ if such a short exact sequence exists. We say that such an extension of $G$ \emph{integrates} a central extension of Lie algebras \eqref{eq:alg-ce} if $\hatG$ is a Lie group with Lie algebra $\hatfg$, $K$ is a one-dimensional Lie group (whose Lie algebra we identify with $\RR$) and the maps appearing in \eqref{eq:alg-ce} are the derivatives of those in \eqref{eq:grp-ce}.

The potential obstruction to Kirillov's theorem (our Theorem~\ref{thm:homo-sym-ext}) holding in the case of a group $G$ which is not simply-connected is the fact that a one-dimensional central extension of $G$ integrating the central extension of Lie algebras \eqref{eq:alg-ce} defined by the cocycle $c$ arising from the moment map is not guaranteed to exist a priori. 

This problem might be evaded by replacing $G$ by its universal cover $\widetilde G$, for which such a central extension always exists, but this is not satisfactory if one is interested, for example, in classifying the homogenous symplectic spaces for a fixed (connected) Lie group $G$. We now state a theorem by Neeb (which is an improvement on an earlier result by Tuynman and Wiegerinck \cite{MR948561}) which we will use to show the existence of the appropriate central extension.

To a 2-cocycle $c$ of $\fg$ considered as an element $c\in\bigwedge^2\fg^*=\bigwedge^2 T^*_eG$, we associate a closed, left-invariant 2-form $\Omega$ on $G$ defined by

\begin{equation}\label{eq:invt-2-form}
	\Omega_g := L_{g^{-1}}^*c
\end{equation}
for all $g\in G$, where $L_{g^{-1}}$ is left-translation by $g^{-1}$. Equivalently, $\Omega$ is the unique such 2-form with value $\Omega_e=c$ at the identity. We will denote by $\lambda_X$ and $\rho_X$ respectively the left- and right-invariant vector fields on $G$ associated to $X\in\fg$, and we note that $\Omega$ can also be defined as the unique left-invariant 2-form on $G$ such that

\begin{equation}\label{eq:Omega-lambda}
	\omega(\lambda_X,\lambda_Y) = c(X,Y)
\end{equation}
for all $X,Y\in\fg$.

\begin{theorem}[Neeb \cite{MR1424633}]\label{thm:neeb}
  Let $G$ be a connected Lie group, $c\in\bigwedge^2\fg^*$ a Chevalley--Eilenberg 2-cocycle of the Lie algebra $\fg$ and $\Omega\in\Omega^2(G)$ the corresponding left-invariant closed 2-form. Then there exists a one-dimensional central extension of Lie groups \eqref{eq:grp-ce} integrating the one-dimensional central extension of Lie algebras \eqref{eq:alg-ce} defined by $c$ if and only if for any $X\in\fg$ there exists a function $\Phi_X\in C^\infty(G)$ satisfying
  
  \begin{equation}\label{eq:neeb}
    \imath_{\rho_X}\Omega = d\Phi_X.
  \end{equation}
\end{theorem}

Note that this theorem does not depend on the global topology of the one-dimensional Lie group $K$; it can be freely chosen. As such, we may omit from the discussion in what follows.

\section{Homogeneous symplectic spaces are covers of coadjoint orbits}

We now state and prove our main theorem.

\begin{theorem}[AB, Figueroa-O'Farrill \cite{Beckett:2022wvo}]\label{thm:homo-sym-ext}
  Let $G$ be a connected Lie group and $(M,\omega)$ be a
  simply-connected symplectic homogeneous manifold of $G$.  Then
  there exists a one-dimensional central extension $\widehat{G}$ of
  $G$ such that $M$ is the universal cover of a coadjoint orbit of
  $\widehat{G}$.
\end{theorem}

\begin{proof}
	The assumption that $M$ is simply-connected guarantees the existence of a (not necessarily equivariant) moment map; indeed, since $d\omega=0$, by Cartan's magic formula and $G$-invariance of $\omega$ we have $d\imath_{\xi_X}\omega=\eL_{\xi_X}\omega=0$, so $\imath_{\xi_X}\omega$ is a closed one-form and thus exact since $\pi_1(M)=1$. Thus there exists a comoment map $\varphi:\fg\to C^\infty(M)$, $X\mapsto\phi_X$ satisfying the equation $\imath_{\xi_X} \omega = d\varphi_X$ and a moment map $\mu:M\to\fg^*$ dual to $\phi$. Let $\theta$, $c$ be the corresponding cocycles as described in the introduction and $\Omega$ the closed, left-invariant 2-form given by equation~\eqref{eq:invt-2-form}. We define a map $\Phi:\fg\to C^\infty(G)$ sending $X\mapsto \Phi_X$, where $\Phi_X(g)=-\pair{\theta(g)}{X}$ for $X\in\fg$ and $g\in G$, and we will show that equation~\eqref{eq:neeb} holds. Since right-invariant vector fields span the tangent space at every point of $G$, it is sufficient to show that
	
	\begin{equation}
		\Omega(\rho_X,\rho_Y) = \eL_{\rho_Y}\Phi_X
	\end{equation}
	for all $X,Y\in\fg$. Indeed, noting that
	$\left( \rho_X \right)_g=\left(\lambda_{\Ad_{g^{-1}}X}\right)_g$
	at any $g\in G$, and using \eqref{eq:Omega-lambda}, we have
	
	\begin{equation}
		\Omega(\rho_X,\rho_Y)(g)
			= \Omega(\lambda_{\Ad_{g^{-1}}X},\lambda_{\Ad_{g^{-1}}Y})(g)
			= c(\Ad_{g^{-1}}X,\Ad_{g^{-1}}Y)
			= (\Ad^*_g c)(X,Y),
	\end{equation}
	where $(\Ad_g^*c)(X,Y):=c(\Ad^*_{g^{-1}}X,\Ad^*_{g^{-1}}Y)$. On the other hand, using a property of the exponential map and the cocycle condition~\eqref{eq:theta-cocycle} for $\theta$,
	
	\begin{equation}
	\begin{split}
		(\eL_{\rho_Y}\Phi_X)(g)
		&= -\dv{t} \pair{\theta(\exp(tY)g)}{X}\eval_{t=0}\\
		&= -\dv{t} \pair{\theta\qty(g\exp(t\Ad_{g^{-1}}Y))}{X}\eval_{t=0}\\
		&= -\dv{t} \qty(\pair{\Ad^*_g\theta\qty(\exp(t\Ad_{g^{-1}}Y))}{X} + \pair{\theta(g)}{X})\eval_{t=0}\\
		&= -\dv{t} \pair{\theta\qty(\exp(t\Ad_{g^{-1}}Y))}{\Ad_{g^{-1}}X}\eval_{t=0}\\
		&= (\Ad^*_gc)(X,Y).
	\end{split}
	\end{equation}
	Thus equation~\eqref{eq:neeb} holds, and so by Theorem~\ref{thm:neeb} there exists a central extension \eqref{eq:grp-ce} of $G$ integrating the central extension \eqref{eq:alg-ce} of $\fg$ determined by $c$.
	\\
	The action of $G$ on $M$ induces an action of $\hatG$ on $M$ by $\widehat{g}\cdot m = \pi(\widehat{g})\cdot m$, where $\widehat{g}\in\hatG$ and $\pi:\hatG\to G$ is the quotient map; clearly this action is symplectic and transitive since the action of $G$ is. We denote by $\xi_{(X,u)}$ the fundamental vector field on $M$ generated by $(X,u)\in\hatfg=\fg\oplus\RR$, but $(0,u)\in\ker\pi_*$ for all $u\in\RR$, so $\xi_{(X,u)}=\xi_X$. Now let $\widehat\varphi:\hatfg\to C^\infty(M)$, $(X,u)\mapsto\widehat\varphi_{(X,u)}$ be the map given by $\widehat\varphi_{(X,u)}(p)=\varphi_X(p)+u$ for all $p\in M$. We have
	
	\begin{equation}
		d\widehat\varphi_{(X,u)} = d\varphi_X = \imath_{\xi_X}\omega = \imath_{\xi_{(X,u)}}\omega,
	\end{equation}
	so $\widehat\varphi$ is a comoment map for the action of $\hatG$ on $M$, and
	
 	\begin{equation}
 		\widehat\varphi_{\comm{(X,u)}{(Y,v)}} = \widehat\varphi_{(\comm{X}{Y},c(X,Y))} = \varphi_{\comm{X}{Y}} + c(X,Y) = \acomm{\varphi_X}{\varphi_Y} = \acomm{\widehat\varphi_{(X,u)}}{\widehat\varphi_{(Y,v)}}
 	\end{equation}
 	so $\widehat\varphi$ is a homomorphism of Lie algebras, and so since $G$ is connected, the dual moment map $\widehat\mu:M\to\hatfg^*$ is $\hatG$-equivariant (with  respect to the adjoint action on $\hatfg$) and so $(M,\omega,\widehat\mu)$ is a hamiltonian $\hatG$-space. It then follows from Corollary \ref{cor:covering} that $M$ is the universal cover of a coadjoint orbit $\widehat\bigO$ of $\hatG$ with equivariant symplectic covering map $\widehat\mu$.
\end{proof}

\section{Affinisation of the coadjoint action}
\label{sec:affine}

As first noted by Souriau, although a moment map $\mu:M\to\fg^*$ for a symplectic $G$-space $(M,\omega)$ need not be equivariant with respect to the coadjoint action, there always exists an \emph{affine} action of $G$ on $\fg^*$, $\rho:G\to \Aff(\fg^*)$, for which $\mu$ is automatically equivariant; the action of $g\in G$ on $\alpha\in\fg^*$ is given in our conventions by

\begin{equation}
	\rho(g)\alpha = \Ad^*_g\alpha - \theta(g)
\end{equation}
where $\theta$ is the Souriau symplectic cocycle. That this formula defines an action is equivalent to $\theta$ being a cocycle of $G$ with values in the coadjoint representation. We now interpret Theorem~\ref{thm:homo-sym-ext} in terms of this action.

Generalising the KKS structure on coadjoint orbits, there is a natural $G$-invariant symplectic structure on orbits of this affine action. On the orbit $\bigO^{\mathrm{aff}}_\alpha=\{\rho(g)\alpha\,|\,g\in G\}$ of $\alpha\in\fg^*$, this is given by

\begin{equation}
	(\omega_{\mathrm{aff}})_\alpha(\zeta_X,\zeta_Y) = \pair{\alpha}{\comm{X}{Y}} + c(X,Y)
\end{equation}
for all $X,Y\in\fg$, where $\zeta_X$ is the fundamental vector field on $\bigO^{\mathrm{aff}}_\alpha$ with respect to the affine action associated to $X\in\fg$, and we note that it is sufficient to give the value at $\alpha$ by $G$-equivariance. The inclusion $i:\bigO^{\mathrm{aff}}_\alpha\to\fg^*$ is a moment map as in the case of coadjoint orbits, but $(\bigO^{\mathrm{aff}}_\alpha,\omega_{\mathrm{aff}},i)$ is not hamiltonian for $\theta\neq 0$.

Such affine orbits are closely related to coadjoint orbits of $\hatG$. Indeed, note that the vector space splitting $\hatfg=\fg\oplus\RR$ induces a dual splitting $\hatfg^*=\fg^*\oplus\RR$; letting $(\alpha,\zeta)\in\hatfg^*$ and $\widehat{g}\in \hatG$ with image $g\in G$ under the quotient map and denoting the coadjoint representation of $\hatG$ by $\widehat\Ad^*$, one can show that

\begin{equation}
	\widehat\Ad^*_{\widehat{g}}(\alpha,\eta)=(\Ad_g^*\alpha-\eta\theta(g),\eta).
\end{equation}
Since the right-hand side of this equation only depends on $\widehat{g}$ through $g$, we see that this coadjoint action factors to an action of $G$ on $\hatfg^*$ which we denote by $\widehat\Ad^*_g(\alpha,\eta)=(\Ad_g^*\alpha-\eta\theta(g),\eta)$ and that both actions preserve hyperplanes $\eta=\text{const.}$. In particular, the action on the $\eta=1$ hyperplane is $\widehat\Ad^*_g(\alpha,1)=(\rho(g)\alpha,1)$, so there is an isomorphism of $G$-spaces $\bigO^{\mathrm{aff}}_\alpha\to\widehat\bigO_{(\alpha,1)}$ where the latter space is the coadjoint orbit of $(\alpha,1)$. One can show that this map also preserves the symplectic structures, so the two orbits are isomorphic as homogeneous $G$-spaces.

We can thus interpret Theorem~\ref{thm:homo-sym-ext} as saying that any simply-connected homogeneous symplectic $G$-space $(M,\omega)$ is the universal cover of an orbit $\bigO^{\mathrm{aff}}$ of the affine action with the (symplectic, $G$-equivariant) covering map being the moment map $\mu:(M,\omega)\to \bigO^{\mathrm{aff}}$.
 
\section{Invariant presymplectic structures on $G$ and integrability}

We now note that we can interpret the closed left-invariant 2-form $\Omega$ on $G$ associated to a 2-cocycle $c$ of $\fg$ as a \emph{presymplectic} structure\footnote{That is, $\omega$ is a closed 2-form of constant rank which may be degenerate.} on $G$ which is invariant under the action of $G$ on itself by left translation. Under this interpretation, the criterion \eqref{eq:neeb} says that $\Phi:\fg\to C^\infty(G)$ is a comoment map -- this perspective is not emphasised by Neeb but is given in the earlier work \cite{MR948561} on which Neeb builds. We showed in the proof of Theorem~\ref{thm:homo-sym-ext} that such a map $\Phi$ can be constructed from a symplectic cocycle $\theta$ integrating $c$ by setting $\Phi_X(g)=-\pair{\theta(g)}{X}$ for $X\in\fg$ and $g\in G$. The converse is also true; if the comoment map exists, we can assume without loss of generality that $\Phi_X(e)=0$ for all $X\in\fg$,\footnote{Indeed, if we have $\Phi_X(e)\neq 0$ then the map $\Phi':\fg\to C^\infty(G)$ given by $\Phi'_X(g)=\Phi_X(g)-\Phi_X(e)$ satisfies $\Phi'_X(e)=0$ and $d\Phi_X'=d\Phi_X=\imath_{\rho_X}\Omega$.} and then the same formula defines a map $\theta:G\to\fg^*$ which one can show is a symplectic cocycle for $G$ which integrates $c$. We have thus shown that the existence of a symplectic cocycle integrating $c$ is equivalent to the existence of the comoment map $\Phi$, and so we may restate Theorem~\ref{thm:neeb} as follows.
 
\begin{theorem}\label{thm:group-ext-symp-cocycle}
  Let $G$ be a connected Lie group and \eqref{eq:alg-ce} a one-dimensional central extension of the Lie algebra $\fg$. This integrates to a one-dimensional central extension \eqref{eq:grp-ce} of the group $G$ if and only if there exists a symplectic cocycle   $\theta:G\to\fg^*$ integrating a representative 2-cocycle $c$ of $\fg$ of the cohomology class corresponding to the extension  \eqref{eq:alg-ce}.
\end{theorem}

\section{Conclusion}

The classical theorem of Kirillov, our  Theorem~\ref{thm:homo-sym-ext}, that a simply-connected homogeneous symplectic space of a connected Lie group $G$ is the universal cover of a coadjoint orbit either of $G$ or of a one-dimensional central extension of $G$ -- sometimes paraphrased as "every homogeneous symplectic space is locally a coadjoint orbit" -- was previously only proved in the literature under the assumption that $G$ is simply-connected. In the proof of Theorem~\ref{thm:homo-sym-ext}, we showed that this assumption can be removed. In Section~\ref{sec:affine}, we also interpreted the theorem as saying that simply-connected symplectic homogeneous $G$-spaces are universal covers of \emph{affine} $G$-orbits in $\fg^*$.

The key to removing the assumption that $G$ is simply-connected was Neeb's Theorem~\ref{thm:neeb} on the integrability of one-dimensional central extensions of Lie algebras, which we reinterpreted as a statement about the existence of a comoment map for a left-invariant presymplectic form on $G$ and then as a statement about integrability of cocycles (Theorem~\ref{thm:group-ext-symp-cocycle}). We thus emphasise the intimate links between one-dimensional central extensions of Lie groups and algebras, the existence of moment maps for homogeneous symplectic spaces, and symplectic cohomology.

\reftitle{References}

\bibliography{main}

\end{document}